\documentclass[a4paper,12pt]{article}
\usepackage[english]{babel}
\usepackage[utf8]{inputenc}
\usepackage{csquotes}
\usepackage[square,sort,comma,numbers]{natbib}
\usepackage{subcaption}
\usepackage{color}

\usepackage{amssymb}
\usepackage{amsmath}
\usepackage{amsthm}
\usepackage{exscale}
\usepackage{bm}
\usepackage{mathtools}

\usepackage{enumerate}
\usepackage{graphicx}
\usepackage{multirow}
\usepackage{rotating} 

\usepackage[hidelinks]{hyperref}

\DeclareMathOperator\supp{supp}
\DeclareMathOperator\law{Law}
\newcommand{\dd}{\text{d}}
\newcommand{\av}[1]{\left<#1\right>}
\renewcommand{\P}{\mathbb P}
\newcommand{\R}{\mathbb{R}}
\newcommand{\E}{\mathbb{E}}
\newcommand{\N}{\mathbb{N}}

\newcommand{\F}{\mathcal F}
\newcommand{\C}{\mathcal C}
\newcommand{\cS}{\mathcal{S}}
\newcommand{\cM}{\mathcal{M}}
\newcommand{\cB}{\mathcal{B}}
\newcommand{\cL}{\mathcal{L}}
\newcommand{\cH}{\mathcal{H}}
\newcommand{\Schw}{\mathcal{S}}

\newcommand{\Itos}{It\^o's\ }
\newcommand{\I}{\mathbb{I}}

\newtheorem{theorem}{Theorem}[section]
\newtheorem{corollary}[theorem]{Corollary}
\newtheorem{lemma}[theorem]{Lemma}
\newtheorem{prop}[theorem]{Proposition}

\theoremstyle{definition}
\newtheorem{definition}[theorem]{Definition}

\newtheorem{example}[theorem]{Example} 

\numberwithin{equation}{section}

\graphicspath{{Figures/}}
\title{Dean--Kawasaki equation with initial condition in the space of positive distributions}
\author{Vitalii Konarovskyi\footnote{Faculty of Mathematics, Informatics and Natural Sciences, University of Hamburg, 20146 Hamburg, Germany. E-mail: \href{mailto:vitalii.konarovskyi@uni-hamburg.de}{vitalii.konarovskyi@uni-hamburg.de}} \footnote{Institute of Mathematics of NAS of Ukraine, 01024 Kyiv, Ukraine} \and Fenna Müller\footnote{Fakult\"{a}t f\"{u}r Mathematik und Informatik, Universit\"{a}t Leipzig, 04109 Leipzig, Germany. E-mail: \href{mailto:fenna.mueller@math.uni-leipzig.de}{fenna.mueller@math.uni-leipzig.de}} \footnote{Max Planck Institute for Mathematics in the Sciences, 04103 Leipzig, Germany}}

\begin{document}
\maketitle

\begin{abstract}
    We show that the Dean--Kawasaki equation does not admit non-trivial solutions in the space of tempered measures. 
    More specifically, we consider martingale solutions taking values, and with initial conditions, in the subspace of measures admitting infinite mass and satisfying some integrability conditions.  
    Following work by the first author, Lehmann and von Renesse~\cite{konarovskyi2019dean}, we show that the equation only admits solutions if the initial measure is a discrete measure. Our result extends the previously mentioned works by allowing measures with infinite mass.
\end{abstract}

\section{Introduction}
The Dean--Kawasaki equation was proposed by James Dean \cite{dean1996langevin} and Kyozi Kawasaki \cite{kawasaki1973} to describe the density $\rho$ of fluctuating particles at inverse temperature $\beta$.
\begin{align*}
        \partial_t \rho &=  \beta^{-1} \Delta \rho 
        + \nabla \cdot\left( \rho \nabla \frac{\delta F}{\delta \rho}[\rho] \right)
        + \nabla \cdot \left( \sqrt{\rho} \eta \right).
\end{align*}
In this nonlinear, singular stochastic partial differential equation, $F$ describes the interaction of the particles within the fluid and $\eta$ is a vector of space-time white noise.
Phrasing the Dean--Kawasaki equation as a martingale problem in the space of finite measures, in \cite{konarovskyi2020dean,konarovskyi2019dean} it was shown that its solutions exist only for initial condition of the form $\frac{1}{\alpha} \sum_{k= 1}^n \delta_{x_i}$ for some $n \in \N$. Furthermore, these solutions are then given by an empirical measure of particles $X_t^i$ following a system of coupled stochastic differential equations with initial conditions $X_0^i = x_i$ \cite{konarovskyi2020dean}.
It is then natural to ask whether this rigidity still holds when allowing solutions and initial conditions with infinite mass. 
In this article we obtain an analogous result in the space of positive tempered distributions. 
Specifically, we show that even allowing initial conditions with infinite mass, e.g. the Lebesgue measure, the equation only admits solutions if the initial condition is a sum of Dirac masses.
We work on the space of tempered measures $\cM_\Schw$ on $\R^n$, i.e. the space of measures which integrate all Schwartz functions $\Schw$.
Here, we are considering solutions to a simplified version without interaction
\begin{align}
        \partial_t \rho &=  \frac{\alpha}{2} \Delta \rho
        + \nabla \cdot \left( \sqrt{\rho} \eta \right)
        \label{eqn:DK}
\end{align}
with $\alpha > 0$.
In this setting the definition of a solution and the main statement read as follows.
\begin{definition}\label{def:MPloc}
	A continuous $\cM_\Schw$-valued process $(\mu_t)_{t \ge 0}$ is a solution to the Dean--Kawasaki equation \eqref{eqn:DK} with initial condition $\nu \in \cM_{\Schw}$ if $\mu_0 = \nu$ and
$$
	M_t(\varphi) =  \av{\mu_t, \varphi} - \av{\mu_0, \varphi} - \int_0^t \frac{\alpha}{2}\av{\mu_s, \Delta \varphi }  \dd s,\quad t\ge 0,
$$
is a martingale with respect to the natural filtration $(\F)_{t\ge0}$ generated by $(\mu_t)_{t\ge 0}$ with quadratic variation
$$
	[M_\cdot(\varphi)]_t =  \int_0^t \av{\mu_s, |\nabla \varphi|^2} \dd s,\quad t\ge 0, 
$$
for each $\varphi \in \Schw$. 
\end{definition}

\begin{theorem}\label{theo:main}
	Let $\nu \in \cM_{\Schw}$. Then the Dean--Kawasaki equation \eqref{eqn:DK} has a unique in law solution $(\mu_t)_{t \ge 0}$ started from $\nu$ if and only if $\nu(\R^d)\alpha \in \N \cup \{\infty\}$ and there exists at most countable family $x_i$, $i \in I \subseteq \N$, such that $\nu = \frac{1}{\alpha} \sum_{i \in I} \delta_{x_i}$.	Moreover, in this case, 
    \begin{equation}\label{eqn:form_of_mu_t}
      \mu_t = \frac{1}{\alpha} \sum_{i \in I} \delta_{B_{\alpha t}^{i}},\quad t\ge 0,
    \end{equation}
    where $(B_{t}^i)_{t\ge 0}$, $i\in I$, is a family of independent $d$-dimensional Brownian motions with $B^i_0 = x_i$ a.s.
\end{theorem}

Note that the choice of the state space $\mathcal{M}_\Schw$ for a solution to the Dean--Kawasaki equation~\eqref{eqn:DK} is motivated by the fact that the process $(\mu_t)_{t\ge 0}$, defined by the equation~\eqref{eqn:form_of_mu_t}, can blow up in finite time if $\mu_0$ is only a locally finite measure (see Example~\ref{exa:blow_up} below). Therefore, this leads to the problem of the existence of solutions. The restriction to the space $\mathcal{M}_\Schw$ provides additional assumptions on the tails of the initial distribution $\mu_0$ that prevent the blowup.

The dichotomy of existence and non-existence does not prevail for the Dean--Kawasaki equation with noise which is white in time but coloured in space. Approximating the square root, solutions to the martingale problem which are absolutely continuous with respect to the Lebesgue measure are shown to exist in e.g. \cite{mauri2010large,djurdevac2022weak}.
Mild solutions to the Dean--Kawasaki equation of kinetic Langevin systems with coloured noise are treated in \cite{cornalba2019regularized,cornalba2020from}.
The well-posedness for a large class of conservative stochastic partial equations with coloured Stratonovich noise is shown in \cite{fehrman2021wellposed}.
Similarly, these equations are treated from the perspective of random dynamical systems in \cite{fehrman2022ergodicity}. The existence and uniqueness of superposition solutions to the Dean--Kawasaki equation with correlated noise were obtained in~\cite{gess2022conservative}.
Although, the original Dean--Kawasaki equation with a smooth drift potential $F$ has only trivial solutions, it can admit complex solutions in the case of singular drift (see, e.g. \cite{von2009entropic,Andres:2010} for $\alpha>0$ and~\cite{Marx:2018,Schiavo:2018,Konarovskyi:AP:2017,Konarovskyi:CFWDPA:2022,Konarovskyi:CPAM:2019,Konarovskyi:CD:2020} for $\alpha=0$).   

Note that the original idea of the physicists' was to describe the density of the fluctuating particles instead of the particles individually. In particular, this makes sense in the case on an intermediary scale of particles, where describing them individually is hard, but fluctuations are visible \cite{kawasaki1973}, as done in fluctuating hydrodynamics \cite{giacomin1999deterministic,spohn1991large}.
The articles by Konarovskyi et al show, however, that this aim is not achieved by the Dean--Kawasaki equation since its solutions are only empirical measures of the individual particles \cite{konarovskyi2019dean,konarovskyi2020dean}.
Nevertheless, the structure-preserving discretisations of the Dean-Kawasaki equation are the right tool to simulate the density fluctuations of interacting diffusing particles, as shown in the works \cite{cornalba2021dean,cornalba2023density}.
Also, the evolution of a system of independent Brownian motions can be appropriately approximated by solutions to regularized Dean--Kawasaki equation \cite{djurdevac2022weak}. 

{\it Content of the paper.} After setting up the necessary notation and spaces in Section \ref{sec:preliminaries}, we shows that the empirical process for an infinite family of independent Brownian motions solves the equation according to Definition~\ref{def:MPloc} in Section \ref{sec:existence}. 
Section \ref{sec:uniqueness} is devoted to the uniqueness in law of solutions to the Dean--Kawasaki equation. We also prove that other solutions, except those obtained in Section~\ref{sec:existence}, do not exist. 
The proofs of both these facts, rely on the Laplace duality, obtained in Proposition~\ref{prop:laplace_duality}.
Finally, Section~\ref{sec:estimates} contains some auxiliary statement about the properties of solutions to the Hamilton-Jacobi equation.

\section{Preliminaries}
\label{sec:preliminaries}
Let $\C(\R^d)$ denote the space of continuous functions on $\R^d$ and $\C_b(\R^d)$ be its subspace of bounded functions. Let $\C_c(\R^d)$ (resp. $\C_c^\infty(\R^d)$) be the space of continuous (resp. smooth) compactly supported functions on $\R^d$. 
We will denote the space of continuously differentiable functions from an interval $I$ to $\R$ by $\C^1(I)$ and  the space of continuous functions from $I$ to a topological space $E$ by $\C(I, E)$.
For $f,g:\R^d\to\R$ we will write shortly $f\le g$ if $f(x) \leq g(x)$ for all $x \in \R^d$.  Set $\N_0: = \N \cup \{0\}$ and $[n] := \{1,..., n\}$.

The space of Schwartz functions on $\R^d$ will be denoted as usual by $\Schw$, that is, 
$$
\Schw:= \left\{ \varphi \in \C^\infty(\R^d)|\  \sup_{x \in \R^d} |x|^n |D^\beta \varphi(x)| < \infty,\ \forall n\in\N_0,\ \beta \in \N_0^d \right\},
$$
where $D^\beta f(x): = \frac{\partial ^{|\beta|}}{\partial x_1^{\beta_1}\ldots \partial x_d^{\beta_d}}f(x)$ and $|\beta|:=\beta_1+\ldots+\beta_d$ for $\beta=(\beta_1,\ldots,\beta_d)$.
We will equip $\cS$ with the locally convex topology induced by the seminorms 
\begin{equation*}
	\|f\|_{\beta, n}: = \sup_{x \in \R^d} |x|^n \left|D^\beta f(x)\right| 	 
\end{equation*}
for all $n \in \N_0$ and $\beta \in \N_0^d$. We also set $\Schw_+:=\{\varphi \in \Schw|\ \varphi \geq 0\}$. The dual space to $\cS$ will be denoted by $\cS'$. 
Let $\cM_{\text{loc}}$ be the space of locally finite measures on $\R^d$ equipped with the topology of vague convergence.
Throughout, we consider the set of {\it tempered measures} $\cM_{\Schw}:= \cM_{\text{loc}} \cap \Schw'$, where we identify each measure $\mu$ with the linear functional $\varphi\mapsto\av{\mu,\varphi}=\int_{\R^d}\varphi(x)\mu(\dd x)$. According to \cite[Lemma~2.3 and Proposition~2.5]{baake2022note}, $\mu\in\cM_{\text{loc}}$ belongs to $\cM_\Schw$ if  $\av{\mu, |\varphi|} < \infty$ for all $\varphi \in \Schw$. It will be convenient to equip the space $\cM_\Schw$ with the weak-$*$ topology of Schwartz functions, i.e. a sequence of tempered measures $\{\mu_n\}_{n \in \N}$ converges in $\cM_{\Schw}$ to a tempered measure $\mu$  if and only if $\lim_{n \to \infty} \av{\mu_n, \varphi} = \av{\mu, \varphi}$ for all $\varphi \in \Schw$.
Note that $\cM_\Schw$ can be identified as the set of positive tempered distributions in the following sense: 
$$
\mathcal{M}_\Schw = \{\rho \in \Schw'|~ \varphi \in \Schw_+ \Rightarrow \av{\rho, \varphi} \geq 0\}.
$$
Thus, $\mathcal{M}_\Schw$ is a closed subset of $\Schw'$.  For more information on the relationship between tempered distributions and tempered measures see \cite{baake2022note}.

For $\varphi:\R^d\to\R$, we define the functions $\varphi_+(x) = \max\{\varphi(x), 0\}$ and $\varphi_-(x) = \max\{-\varphi(x), 0\}$, $x\in\R^d$, and note that $\varphi = \varphi_+ - \varphi_-$.
We denote the indicator function of a set $A$ by $\I_{A}$.

Let $P_t$ be the heat semigroup whose generator is given by $\frac{\alpha}{2}\Delta$ and its carr\'e-du-champs operator by $\Gamma = |\nabla \cdot|^2$. Note that $V_t\varphi = - \alpha\ln\left( P_t e^{-\frac{1}{\alpha}\varphi} \right)$, $t\ge 0$, is the (Cole-Hopf) solution to the Hamilton-Jacobi equation
\begin{equation}\label{eqn:hamilto_jacobi_equation}
\begin{split}
	\partial_t v(t,x) &=  \frac{\alpha}{2} \Delta v(t,x) - \frac{1}{2} \Gamma\left( v(t,x) \right),\\
	v(0,x) &=  \varphi(x),
 \end{split}
\end{equation}
for every $\varphi\in\Schw$. Some bounds on the heat semigroup and on the Cole-Hopf solution may be found in Section \ref{sec:estimates}.

\section{Proof of the main result}

\subsection{Existence of solutions given by empirical measures}
\label{sec:existence}

The goal of this section is to show that the process $(\mu_t)_{t\ge 0}$, defined by~\eqref{eqn:form_of_mu_t}, solves the Dean--Kawasaki equation. We will start from an example which shows that the process defined by~\eqref{eqn:form_of_mu_t} can blow up if $\mu_0$ is only a locally-finite measure without any assumptions on the growth of its tails.

\begin{example}\label{exa:blow_up}
Consider a family of $d$-dimensional independent Brownian motions $(B_t^k)_{t\ge 0}$, $k \in \N$, with initial condition $B_0^k = \sqrt{\ln k}e_1$, where $e_1 = (1, 0, \dots, 0)$. Since the sequence $\{\sqrt{\ln k}\}_{k \in \N}$ does not have a cluster point, $\sum_{k \in \N} \delta_{\sqrt{\ln k} e_1} \in \mathcal{M}_{\text{loc}}$. Then, the probability for one of the particles to be in the unit cube is bounded below by 
\begin{align*}
	\P\left\{B_t^k \in [0,1]^d\right\} &=  \P\left\{ \sqrt{t} \xi \in [0,1] \right\}^{d-1} \P\left\{ \sqrt{t}\left( \xi+\ln k \right) \in [0,1] \right\}\\
	&\geq \frac{1}{\sqrt{2\pi t}^d} e^{-\frac{d-1}{2t}} e^{-\frac{\left(\sqrt{\ln k}\right)^2}{2t}}
	= \frac{1}{\sqrt{2\pi t}^d} e^{-\frac{d-1}{2t}} \frac{1}{k^{\frac{1}{2t}}} , 
\end{align*}
where $\xi$ is a standard Gaussian random variable. Hence,
\begin{align*}
	\sum_{k \in \N}\P\left\{ B_t^k \in [0,1]^d  \right\} 
	\geq \frac{1}{\sqrt{2\pi t}^d} e^{-\frac{d-1}{2t}}\sum_{k \in \N}  \frac{1}{k^\frac{1}{2t}} = \infty
\end{align*}
for all $t \geq \frac{1}{2}$. By the Borel-Cantelli lemma, 
\begin{equation*}
\mathbb{P}\left\{ \mu_t([0,1]) = \infty \right\} = \mathbb{P}\left\{ B_t^k \in [0,1] ~\text{for infinitely many}~ k \right\} = 1
\end{equation*} 
for $t \geq \frac{1}{2}$. 
 
\end{example}

\begin{lemma}\label{lem:existence_of_solution}
	Let $\{x_k\}_{k \in \N}$, be a family of points from $\R^d$ such that $\sum_{k \in \N} \delta_{x_k} \in \mathcal M_{\Schw}$ and  $(B^k_t)_{t\ge 0}$, $k \in \N$, be a family of independent $d$-dimensional  Brownian motions satisfying $B_0^k = x_k$ a.s. for all $k \in \N$.
	Then, the process
    \begin{equation}\label{eqn:measure_mu}
      \mu_t = \frac{1}{\alpha} \sum_{k\in\N} \delta_{B^k_{\alpha t}},\quad t\ge 0,
    \end{equation}
    has a continuous version in $\cM_\Schw$ that is a solution to the Dean--Kawasaki equation~\eqref{eqn:DK} with initial value $\mu_0 = \frac{1}{\alpha}\sum_{k \in \N} \delta_{x_k}$.
\end{lemma}

\begin{proof}
        For $\varphi \in\Schw$ and $n\in\N$ we define a continuous square-integrable martingale $M_t^n(\varphi)$, $t\ge 0$,  by
        \begin{align*}
                M_t^n(\varphi) : =\av{\mu_t^n,\varphi}-\av{\mu_0^n,\varphi}-\frac{\alpha}{2}\int_0^t\av{\mu^n_s,\Delta\varphi}\dd s,
        \end{align*} 
        where $\mu_t^n = \frac{1}{\alpha} \sum_{k=1}^n \delta_{B^k_{\alpha t}}$.     We are going to show that $M_t^n(\varphi)$, $t\in[0,T]$, is a Cauchy sequence in $\C([0,T],\R)$ in probability for each $T>0$.
        Fix $\varepsilon >0$ and define the stopping time 
        $$
          \tau_{n,m} := \inf\left\{ t \in [0,T]|~ |M_t^n(\varphi) - M_t^{n+m}(\varphi)| > \varepsilon \right\} \wedge T
        $$
        with respect to the filtration $(\F_t)_{t\ge0}$ generated by the family of Brownian motions $(B_{\alpha t}^k)_{t\ge 0}$, $k\in\N$. By the Markov inequality and the optional sampling theorem, we have
        \begin{align*}
                &\P\left\{ \sup_{t \in [0,T]} \left| M_t^{n+m}(\varphi) - M_t^{n}(\varphi)\right|\ge \varepsilon \right\}=  \P\left\{ |M_{\tau_{n,m}}^{n+m}(\varphi) - M_{\tau_{n,m}}^{n}(\varphi)| \geq \varepsilon \right\}\\
                &\qquad\leq \frac{1}{\varepsilon^2} \E |M_{\tau_{n,m}}^n(\varphi) - M_{\tau_{n,m}}^{n+m}(\varphi)|^2 
        = \frac{1}{\varepsilon^2} \E [M^{n+m}(\varphi) - M^n(\varphi)]_{\tau_{n,m}} \\
        &\qquad \le \frac{1}{\varepsilon^2} \E \int_0^T \sum_{k = n+1}^{n+m} \left|\nabla \varphi(B_{\alpha s}^k)\right|^2 \dd s
        = \frac{1}{\varepsilon^2} \int_0^T \sum_{k = n+1}^{n+m} P_s |\nabla \varphi|^2(x_k) \dd s.
        \end{align*}
        Note that $\sup_{t\in [0,T]}\sum_{k =1}^{\infty} P_t|\nabla \varphi|^2(x_k)=\sup_{t\in[0,T]}\av{P_t|\nabla \varphi|^2,\mu_0}<\infty$ due to Lemma~\ref{lemma:Schw1}. Thus, by the dominated convergence theorem, we conclude that
        \[
          \int_0^T \sum_{k = n+1}^{n+m} P_s |\nabla \varphi|^2(x_k) \dd s\to 0
        \]
        as $n\to\infty$. Next, by \cite[Lemma B.11]{Cherny:2005} and the fact that $\C\left( [0,T],\R \right)$ is complete, there exists a continuous local martingale $(M_t(\varphi))_{t \ge 0}$ such that for every $ \varepsilon >0$ and $T\ge 0$
\begin{align}\label{eqn:convergence_of_M_n}
        \lim_{n \to \infty} \P\left\{\sup_{t \in [0,T]} |M_t^n(\varphi) - M_t(\varphi)| > \varepsilon \right\} = 0
\end{align}
and
\begin{align*}
        \lim_{n \to \infty} \P\left\{\sup_{t \in [0,T]} |[M^n(\varphi)]_t - [M(\varphi)]_t| > \varepsilon \right\} = 0.
\end{align*}

To identify the limit, consider the terms individually. Note that for each $\varphi \in \Schw$
\begin{align*}
    \E \av{\mu_t^n, |\varphi|} =\frac{1}{\alpha} \sum_{k \in \N} \E|\varphi(B_{\alpha t}^k)| = \av{\mu_0^n, P_t |\varphi|} \le  \av{\mu_0, P_t |\varphi|}< \infty,
\end{align*}
by Lemma~\ref{lemma:Schw1}.
Thus, $\av{\mu_t^n, \varphi}\to\av{\mu_t,\varphi}$ a.s. as $n\to\infty$, where $\mu_t$ is defined by the equality~\eqref{eqn:measure_mu}. Setting $\psi = \Delta \varphi\in\Schw$, we also conclude that 

\begin{align}
\begin{split}
        \E \sup_{t \in [0,T]} \left(\int_0^t \av{\mu_s, \psi} - \av{\mu_s^n, \psi} \dd s\right)^2&
        \leq  T\E\sup_{t \in [0,T]} \int_0^t (\av{\mu_s, \psi} - \av{\mu_s^n, \psi})^2 \dd s\\
        &\le T\E \int_0^T (\av{\mu_s, \psi} - \av{\mu_s^n, \psi})^2 \dd s\to 0
\end{split}
     \label{eqn:bv}
\end{align}
as $n\to \infty$. Indeed, using Fatou's lemma, the convergence $\av{\mu_s^n,\psi}\to\av{\mu_s,\psi}$ and then the monotone convergence theorem, we get
\begin{align*}
\E \int_0^T (\av{\mu_s, \psi} &- \av{\mu_s^n, \psi})^2 \dd s
\leq\liminf_{n\to\infty} \E \int_0^T (\av{\mu_s^{n+m}, \psi} - \av{\mu_s^n, \psi})^2\dd s\\
&=\frac{1}{\alpha^2}\liminf_{n\to\infty}\E\int_0^T \sum_{k,l = n+1}^{n+m} \psi(B_{\alpha s}^k)\psi(B_{\alpha s}^l) \dd s
 \\
 &\le \frac{1}{\alpha^2}\liminf_{n\to\infty}\int_0^T \sum_{k=n+1}^{n+m} P_s \psi^2(x_k) \dd s+\frac{1}{\alpha^2}\liminf_{n\to\infty}\int_0^T \left(\sum_{k = n+1}^{n+m} |P_s\psi(x_k)|\right)^2 \dd s\\
 &=\frac{1}{\alpha^2}\int_0^T \sum_{k=n+1}^{\infty} P_s \psi^2(x_k) \dd s+\frac{1}{\alpha^2}\int_0^T \left(\sum_{k = n+1}^{\infty} |P_s\psi(x_k)|\right)^2 \dd s.
\end{align*}
Since 
$$
\sup_{s\in [0,T]}\frac{1}{\alpha} \sum_{k=1}^\infty P_s \psi^2(x_k)=\sup_{s\in[0,T]}\av{\mu_0,P_s(\psi^2)}<\infty
$$ 
and 
$$
\sup_{s\in [0,T]}\frac{1}{\alpha} \sum_{k=1}^\infty|P_s\psi(x_k)|=\sup_{s\in [0,T]}\av{\mu_0,|P_s\psi|}<\infty
$$ 
due to Lemma~\ref{lemma:Schw1} and Corollary~\ref{corollary}, the dominated convergence theorem implies the convergence in~\eqref{eqn:bv}.

Hence, for each $t\ge 0$
        \begin{align*}
                M_t(\varphi) = \av{\mu_t, \varphi}  - \av{\mu_0, \varphi} - \frac{\alpha}{2} \int_0^t \av{\mu_s, \Delta \varphi} \dd s\quad\mbox{a.s.}
        \end{align*}
Similarly, one can show that the quadratic variation of $M_\cdot(\varphi)$ is given by
\[
[M(\varphi)]_t = \int_0^t \av{\mu_s,|\nabla\varphi|^2} \dd s,\quad t\ge 0.
\]

To complete the proof of the lemma, we only need to show that the process $(\mu_t)_{t\ge 0}$ has a continuous version in $\cM_\Schw$. According to the observation above, for every $m\in\N$, set of functions $\varphi_i\in\Schw$, $i\in[m]$, and $t_i\ge 0$, $i\in[m]$, the sequence of random vectors $(\langle \mu^n_{t_1},\varphi_1\rangle,\ldots,\langle \mu^n_{t_m},\varphi_m\rangle)$, $n\in\N$, converges almost surely to $(\langle \mu_{t_1},\varphi_1\rangle,\ldots,\langle \mu_{t_m},\varphi_m\rangle)$ in $\R^m$ as $n\to\infty$. Therefore, it converges in distribution in $\R^m$. Thus, to show the existence of a continuous modification of $(\mu_t)_{t\ge 0}$, it is enough to check the tightness of the family $(\langle\mu^n_t,\varphi\rangle)_{t\in[0,T]}$, $n\in\N$, in $\C([0,T],\R)$ for every $\varphi\in\Schw$ and $T>0$, by Theorem 5.3~\cite{Mitoma:1983}.
Using~\eqref{eqn:bv},~\eqref{eqn:convergence_of_M_n} and Markov's inequality, we get for each $\varepsilon>0$
\[
\lim_{n \to \infty} \P\left\{ \sup_{t \in [0,T]}|\av{\mu_s, \varphi} - \av{\mu_s^n, \varphi}| > \varepsilon \right\} = 0.
\]
Hence, $(\av{\mu_t^n, \varphi})_{t \in [0,T]}$ converges to $(\av{\mu_t, \varphi})_{t \in [0,T]}$ in $\C([0,T], \R)$ in probability, and, thus, in distribution. Using Prohorov's theorem, we obtain the required tightness of $(\av{\mu^n_t,\varphi})_{t\in[0,T]}$, $n\in\N$. 
\end{proof}

\subsection{Uniqueness in Law and Ill-posedness}
\label{sec:uniqueness}
The goal of this section is to show that the Dean--Kawasaki equation admits unique $\cM_\Schw$-valued solutions that are defined by~\eqref{eqn:measure_mu}. We will start from auxiliary statements. 

We fix an arbitrary strictly positive function $\kappa\in\Schw$ such that $\kappa(x)=e^{-|x|}$ for all $x\in\R^d$ with $|x|>1$.
\begin{lemma} 
	Let $(\mu_t)_{t \ge 0}$ be a solution to the Dean--Kawasaki equation~\eqref{eqn:DK}. Then $\sup_{s\in[0,t]}\E\av{\mu_s, \kappa}^2 < \infty$ holds for every $t \ge 0$. 
	\label{lemma:technical2}
\end{lemma}

\begin{proof}
We define the family of stopping-times
$$
\tau_n := \inf\{t \geq 0| \av{\mu_t,\kappa} \geq n\},\quad n\in\N.
$$
Then, using \Itos formula, we get
\begin{align*}
    \av{\mu_{t\wedge \tau_n}, \kappa}^2&=\av{\mu_0,\kappa}^2+\alpha\int_0^{t\wedge\tau_n}\av{\mu_{s}, \kappa}\av{\mu_s,\Delta\kappa}\dd s\\ &+\int_0^{t\wedge\tau_n}\av{\mu_s,|\nabla\kappa|^2}\dd s+2\int_0^{t\wedge\tau_n}\av{\mu_s,\kappa}dM_s(\kappa)
\end{align*}
for each $t\ge 0$. Due to the choice of the function $\kappa$, there exists a constant $C>0$ such that 
\[
|\nabla \kappa(x)|^2\le C\kappa(x)\quad \mbox{and}\quad |\Delta\kappa(x)|\le C\kappa(x)
\]
for all $x\in\R^d$. Therefore, using the optimal sampling theorem and the trivial inequality $\av{\mu_{s}, \kappa}\le\av{\mu_{s}, \kappa}^2+1$, we can estimate
\begin{align*}
    \E\av{\mu_{t\wedge \tau_n}, \kappa}^2&\le\av{\mu_0,\kappa}^2+ \alpha C\E\int_0^{t\wedge\tau_n}\av{\mu_{s}, \kappa}^2\dd s+C\E\int_0^{t\wedge\tau_n}\av{\mu_s,\kappa}\dd s\\    &\le\av{\mu_0,\kappa}^2+t+ (1+\alpha)C\int_0^{t}\E\av{\mu_{s\wedge\tau_n}, \kappa}^2\dd s.
\end{align*}
Therefore, 
\[
\E\av{\mu_{t\wedge \tau_n},\kappa}\le \left(\av{\mu_0,\kappa}^2+t\right)e^{(1+\alpha)Ct}
\]
for all $t\ge 0$, by Gr\"onwall's inequality. Using now Fatou's lemma and the continuity of $(\av{\mu_t,\kappa})_{t\ge 0}$, we get
\begin{align*}
	\E\av{\mu_t, \kappa}^2
	\leq \liminf_{n \to \infty}\E\av{\mu_{t\wedge\tau_n}, \kappa}^2 
	\leq \left(\av{\mu_0,\kappa}^2+t\right)e^{(1+\alpha)Ct}.
\end{align*}
for all $t\ge 0$. This completes the proof of the lemma.
\end{proof}

We will next show that the martingale problem in the Definition~\ref{def:MPloc} of the Dean--Kawasaki equation can be extended to time dependent test functions. This will be needed for the proof of the Laplace duality \eqref{eqn:Laplace} below. 

\begin{lemma}
   Let $(\mu_t)_{t \ge 0}$ be a solution to equation~\eqref{eqn:DK} and a function $(\psi_t)_{t\ge 0} \in \C([0,\infty),\Schw)$ satisfying the following assumptions
   \begin{enumerate}[a)]
	\item for each $x\in\R^d$ the function $(\psi_t(x))_{t\ge 0}$ is differentiable and its derivative $(\partial_t\psi_t)_{t\ge 0}$ belongs to $\C([0,\infty),\Schw)$; 
	\item For each $T>0$ there exists a constant $C>0$ such that 
    \[      |\partial_t\psi_t|+|\nabla\psi_t|^2+|\Delta\psi_t|\le C\kappa
    \]
    for all $t\in[0,T]$.
\end{enumerate}
	 Then, the process
	\[
		M_t :=  \av{\mu_t, \psi_t} - \av{\mu_0, \psi_0} - \int_0^t \av{\mu_s, \partial_s \psi_s+\frac{\alpha}{2} \Delta \psi_s } \dd s,\quad t\ge 0,
    \]
	is a continuous square-integrable martingale with quadratic variation
    \begin{equation}\label{eqn:quadratic_variation_of_M}
		[M]_t =  \int_0^t \av{\mu_s, |\nabla\psi_s|^2} \dd s,\quad t\ge 0.
	\end{equation}
	\label{lemma:ito}
\end{lemma}

\begin{proof}
The idea of the proof of the lemma is to approximate the function $\psi$ by a piece-wise constant functions of the form 
\[
\psi^n_t:=\sum_{k=1}^\infty \psi_{t^n_k}\I_{(t^n_{k-1},t^n_k]},\quad t\ge 0,
\]
for the partition $t_k^n:=\frac{k}{n}$, $k\in\N_0$, and then use the martingale problem for $\left(\av{\mu_t,\psi_{t^n_k}}\right)_{t\ge 0}$, $k\in\N_0$. We can write for each $n\in\N$ and $t\ge 0$ 
\begin{align*}
	&\av{\mu_t, \psi_t} - \av{\mu_0, \psi_0} = \sum_{k=1}^{\infty}\left( \av{\mu_{t^n_k \wedge t}, \psi_{t^n_k \wedge t}}- \av{\mu_{t^n_{k-1}\wedge t}, \psi_{t^n_{k-1}\wedge t}} \right)\nonumber\\
&\qquad=\sum_{k=1}^{\infty} \left(\av{\mu_{t_{k-1}^n\wedge t}, \psi_{t_{k}^n\wedge t}}-\av{\mu_{t_{k-1}^n\wedge t}, \psi_{t_{k-1}^n\wedge t}} \right)+\frac{\alpha}{2} \sum_{k=1}^\infty \int_{t_{k-1}^n\wedge t}^{t_k^n\wedge t} \av{\mu_s, \Delta\psi_{t^n_{k}}} \dd s\\
&\qquad+\sum_{k=1}^\infty \left(M_{t_{k}^n\wedge t}(\psi_{t_k^n\wedge t})-M_{t_{k-1}^n\wedge t}(\psi_{t_k^n\wedge t})\right).
\end{align*}
Note that in the summations above only a finite number of terms is not vanishing.
Using the fact that the function $(\psi_t(x))_{t\ge 0}$ is continuously differentiable for each $x\in\R^d$, $|\partial_s\psi_s|\le C\kappa$ for each $s\in[0,t]$ and Fubini's theorem, we get
\begin{align*}
    I_1^n(t):&=\sum_{k=1}^{\infty} \left(\av{\mu_{t_{k-1}^n\wedge t}, \psi_{t_{k}^n\wedge t}}-\av{\mu_{t_{k-1}^n\wedge t}, \psi_{t_{k-1}^n\wedge t}} \right)\\
    &=\sum_{k=1}^{\infty}\int_{t_{k-1}^n\wedge t}^{t_{k}^n\wedge t}\av{\mu_{t_{k-1}^n},\partial_s\psi_s}\dd s \\
    &=\int_0^t\av{\tilde\mu_s^n,\partial_s\psi_s}\dd s,
\end{align*}
where $\tilde\mu_s^n=\sum_{k=1}^\infty\mu_{t_{k-1}^n}\I_{(t_{k-1}^n,t_k^n]}$. Due to the continuity of $(\mu_t)_{t\ge 0}$, we conclude that $\tilde\mu_s^n\to\mu_s$ in $\cM_\Schw$ as $n\ge 0$ a.s. for each $s\in[0,t]$. Note that 
\[
\sup_{s\in[0,t]}|\av{\tilde\mu_s^n,\partial_s\psi_s}|\le C\sup_{s\in[0,t]}\av{\tilde\mu_s^n,\kappa}\le C\sup_{s\in[0,t]}\av{\mu_s,\kappa}<\infty\quad \mbox{a.s.}
\]
Thus, by the dominated convergence theorem, we get for each $T>0$
\[
\sup_{t\in[0,T]}\left|I_1^n(t)- \int_0^t\av{\mu_s,\partial_s\psi_s}\dd s\right|\le \int_0^T\left|\av{\tilde\mu_s^n,\partial_s\psi_s}-\av{\mu_s,\partial_s\psi_s}\right|\dd s\to0 \quad \mbox{a.s.}
\]
as $n\to\infty$. Thus, $\{I_1^n\}_{n\ge 1}$ converges to $\int_0^t\av{\mu_s,\partial_s\psi_s}\dd s$, $t\ge 0$, in $\C([0,\infty),\R)$ a.s.

Similarly, we can show that 
\begin{align*}
    I_2^n(t):&=\sum_{k=1}^\infty \int_{t_{k-1}^n\wedge t}^{t_k^n\wedge t} \av{\mu_s, \Delta\psi_{t^n_{k}}} \dd s=\int_0^t\av{\mu_s,\Delta\psi^n_s}\dd s,\quad t\ge 0,
\end{align*}
converges to $\int_0^t\av{\mu_s,\Delta\psi_s}\dd s,$ $t\ge 0$, in $\C([0,\infty),\R)$ a.s. 
as $n\to\infty$, using the fact that $\Delta\psi^n_s\to\Delta\psi_s$ in $\Schw$ and $\sup_{s\in [0,t]}|\Delta\psi^n_s|\le\sup_{s\in [0,t]}|\Delta\psi_s|\le C\kappa$.

Since 
\[
M_{t_{k}^n\wedge t}(\psi_{t_k^n\wedge t})-M_{t_{k-1}^n\wedge t}(\psi_{t_k^n\wedge t})=\begin{cases} 0 & t \leq t_{k-1}^n\\
		M_{t_{k}^n\wedge t}(\psi_{t_k^n}) - M_{t_{k-1}^n}(\psi_{t_k^n})& t > t_{k-1}^n,
		\end{cases}
\]
a simple computation shows that 
\begin{align*}
    I_3^n(t):=\sum_{k=1}^\infty \left(M_{t_{k}^n\wedge t}(\psi_{t_k^n\wedge t})-M_{t_{k-1}^n\wedge t}(\psi_{t_k^n\wedge t})\right),\quad t\ge 0,
\end{align*}
is a continuous martingale with respect to the filtration generated by $(\mu_t)_{t\ge 0}$ with quadratic variation
\begin{align*}
    [I_3^n]_t=\int_0^t \av{\mu_s,|\nabla\psi_s^n|^2}\dd s,\quad t\ge 0.
\end{align*}
According to the equality 
\[
I_3^n(t)=\av{\mu_t, \psi_t} - \av{\mu_0, \psi_0}-I_1^n(t)-\frac{\alpha}{2}I_2^n(t),\quad t\ge 0,
\]
and the convergence of $\{I_1^n\}_{n\ge 1}$ and $\{I_2^n\}_{n\ge 1}$ in $\C([0,\infty),\R)$ a.s., we get that the sequence $\{I_3^n\}_{n\ge 1}$ converges to $(M_t)_{t\ge 0}$ in $\C([0,\infty),\R)$ a.s. as $n\to \infty$. By \cite[Lemma B.11]{Cherny:2005} and the bound $|\nabla \psi_t|^2\le C\kappa$, we conclude that $(M_t)_{t\ge 0}$ is a continuous local martingale with quadratic variation given by~\eqref{eqn:quadratic_variation_of_M}.

The fact that $(M_t)_{t\ge 0}$ is a square-integrable martingale follows from the estimate
\[
\E[M]_t=\E\int_0^t\av{\mu_s,|\nabla \psi_s|^2}\dd s\le C\int_0^t\E\av{\mu_s,\kappa}\dd s
\]
and Lemma~\ref{lemma:technical2}. This completes the proof of the lemma.
\end{proof}

\begin{prop}\label{prop:laplace_duality}
	Let $(\mu_t)_{t \ge 0}$ be a solution to the Dean--Kawasaki equation \eqref{eqn:DK} with initial condition $\nu \in \mathcal M_\Schw$. Then the following Laplace duality
	\begin{equation}
		\mathbb{E} e^{-\av{\mu_t, \varphi}}  = e^{-\av{\nu, V_t \varphi}}
		\label{eqn:Laplace}
	\end{equation}
	 holds for every $\varphi \in \Schw_+$,  where $V_t \varphi = - \alpha \ln\left(P_t e^{-\frac{\varphi}{\alpha}}\right)$, $t \ge 0$, is a solution to the equation \eqref{eqn:hamilto_jacobi_equation}.
\end{prop}

\begin{proof}
	Let $T\ge 0$ and $\varphi \in \C_c^\infty(\R^d)$ be a non-negative function.  By Lemmas~\ref{lemma:technical3} and~\ref{lemma:technical1} below, the function $\psi=(V_{T-t} \varphi)_{t\in[0,T]}$ satisfies the assumptions of Lemma~\ref{lemma:ito}. Thus, using Lemma~\ref{lemma:ito} and \Itos formula, we get for $t\in[0,T]$,
	\begin{align*}
		\dd  e^{-\av{\mu_t,V_{T-t}\varphi}} 
		&=   e^{-\av{\mu_t, V_{T-t}\varphi}} \left( - \av{\mu_t, \Delta V_{T-t}\varphi + \partial_t (V_{T-t}\varphi) }\dd t + \dd [M]_t - \dd M_t \right)\\
		&=  e^{-\av{\mu_t, V_{T-t}\varphi}} \left( - \av{\mu_t, \Delta V_{T-t}\varphi + \partial_t (V_{T-t}\varphi) - |\nabla V_{T-t}\varphi|^2} \dd t - \dd M_t \right)\\
		&=  -e^{-\av{\mu_t, V_{T-t}\varphi}} \dd M_t. 
	\end{align*}
	Due to $e^{-\av{\mu_s,V_{t-s}\varphi}} \leq 1$, the process $(e^{-\av{\mu_t,V_{T-t}\varphi}})_{t \in [0,T]}$ is a continuous martingale. Thus,
    \begin{equation}\label{eqn:laplace_duality_for_compactly_supportef_functions}
    \E e^{-\av{\mu_T,\varphi}}=\E e^{-\av{\mu_T,V_{T-T}\varphi}}=\E e^{-\av{\mu_0,V_{T}\varphi}}=\E e^{-\av{\nu,V_{T}\varphi}}.
    \end{equation}
    
	For an arbitrary non-negative function $\varphi \in\Schw$, there exists a sequence $\{\varphi_n\}_{n \in \N}$ of non-negative functions from $\C_c^\infty(\R^d)$ such that $\{\varphi_n(x)\}_{n\ge 1}$ increases for each $x\in\R^d$ and $\varphi_n \to \varphi$ in $\Schw$. By the dominated convergence theorem and~\eqref{eqn:laplace_duality_for_compactly_supportef_functions}, for each $t>0$
	\begin{align*}
		\E e^{-\av{\mu_t, \varphi}}  =\lim_{n \to \infty}\E e^{-\av{\mu_t, \varphi_n}}=\lim_{n \to \infty}e^{-\av{\nu,V_{t}\varphi_n}}.
    \end{align*}
    On the other hand, applying the monotone convergence theorem and Lemma~\ref{lemma:basic} and using the precise expression of $V_t\varphi_n$, we obtain
    \begin{align*}		
		\lim_{n \to \infty} e^{-\av{\nu,V_t \varphi_n} } = e^{-\av{\nu, \lim_{n \to \infty}V_t \varphi_n}} = e^{-\av{\nu, V_t \varphi}}.
	\end{align*}
 This completes the proof of the proposition.
\end{proof}

Using the Laplace duality, obtained in Proposition~\ref{prop:laplace_duality}, will prove the uniqueness in law of solutions to the Dean--Kawasaki equation.

\begin{prop}
	Solutions to \eqref{eqn:DK} are unique in law.
	\label{prop:uniqueness}
\end{prop}

\begin{proof}
 Let $(\mu_t^1)_{t\ge 0}$ and $(\mu_t^2)_{t\ge 0}$ be solutions to the Dean--Kawasaki equation~\eqref{eqn:DK} started from the same distribution $\nu\in\cM_\Schw$ and defined possibly on different probability spaces $(\Omega^1,\F^1,\P^1)$ and $(\Omega^2,\F^2,\P^2)$, respectively. We will show that their finite distributions coincides, that is, for each $n\in\N$, $0<t_1<t_2<\ldots<t_n$ and $B_i\in\cB(\cM_\Schw)$, $i\in[n]$, the equality
 \[
 \P^1\left\{\mu^1_{t_1}\in B_1,\ldots,\mu^1_{t_n}\in B_n\right\}=\P^2\left\{\mu^2_{t_1}\in B_1,\ldots,\mu^2_{t_n}\in B_n\right\}
 \]
 holds, where $\cB(\cM_\Schw)$ denotes the Borel $\sigma$-algebra on $\cM_\Schw$.

 We first show that one-dimensional distributions of $(\mu^i_t)_{t\ge 0}$, $i\in[2]$, coincide. 
 According to \cite[Proposition~2.1]{Dobrusin_Minlos:1976}, the Borel $\sigma$-algebra on $\Schw'$ coincides with the cylindrical $\sigma$-algebra generated by maps $\mu\mapsto\av{\mu,\varphi}$, $\varphi\in\Schw$. 
 Recall next that every element of $\cM_\Schw$ is a locally-finite measure. Therefore, the map $\mu\mapsto\av{\mu,\varphi}$ is well defined for each $\varphi\in\C_c(\R^d)$. Consequently, $\cB(\cM_\Schw)$ equals to the cylindrical $\sigma$-algebra $\Sigma(\cM_\Schw)$ on $\cM_\Schw$ generated by maps $\mu\mapsto\av{\mu,\varphi}$, $\varphi\in\C_c(\R^d)$, by the approximation argument. 
 Thus, by~\cite[Theorem~3.1]{kallenberg1983random}, it is enough to show that for each $t\ge 0$ and $\varphi\in\C_c(\R^d)$ the random variables $\av{\mu^i_t,\varphi}$, $i\in[2]$, have the same distribution. But this follows from Proposition~\ref{prop:laplace_duality}. Indeed, applying~\cite[Theorem~15.5.1]{kallenberg1983random} to the Laplace transform of the distributions $\law\left(\av{\mu^i,\varphi_1},\ldots,\av{\mu^i_t,\varphi_m}\right)$, $i\in[2]$, and using Proposition~\ref{prop:laplace_duality}, we obtain the equality 
 \begin{equation}\label{eqn:equality_of_laws} \law\left(\av{\mu^1_t,\varphi_1},\ldots,\av{\mu^1_t,\varphi_m}\right)=\law\left(\av{\mu^2_t,\varphi_1},\ldots,\av{\mu^2_t,\varphi_m}\right)
 \end{equation} 
 for each $m\in\N$ and $\varphi_j\in\Schw_+$, $j\in[m]$. According to the approximation argument, the equality~\eqref{eqn:equality_of_laws} is true for non-negative functions $\varphi_j$ from $\C_c(\R^d)$. Hence, $\law\av{\mu^1_t,\varphi}=\law\av{\mu^2_t,\varphi}$ for all $\varphi\in\C_c(\R^d)$ due to the equality $\varphi=\varphi_+-\varphi_-$. 
 
 The equality of finite-dimensional distributions of the processes $(\mu^i_t)_{t\ge 0}$, $i\in[2]$, follows now from the same argument as in the proof of \cite[Theorem~3.4.2]{ethierkurtz}, using the fact that they both are solutions to the same martingale problem.
\end{proof}

We are now ready to prove the main result of this work.
\begin{proof}[Proof of Theorem~\ref{theo:main}]
    The proof is similar to the proof of~\cite[Theorem~2.2]{konarovskyi2019dean}. The only difference is that now we have to deal with locally-finite measures that are tempered distributions and rapidly decreasing test function. This creates additional difficulties which can not be overcome straightforward.

    The uniqueness of solutions was obtained in Proposition~\ref{prop:uniqueness}.  The existence of a solution to the Dean--Kawasaki equation~\eqref{eqn:DK} started from     \begin{equation}\label{eqn:form_of_initial_condition}
    \nu= \frac{1}{\alpha}\sum_{k \in I} \delta_{x_k}
    \end{equation} 
    for some family $x_k\in\R^d$, $k\in I\subset \N$, follows from \Itos formula if $I$ is finite, and is the statement of Lemma~\ref{lem:existence_of_solution} for infinite $I$. It remains only to show that if $(\mu_t)_{t\ge 0}$ is a solution to~\eqref{eqn:DK}, then its initial condition must be given by~\eqref{eqn:form_of_initial_condition}. 

    For fixed $t>0$ and $a_k,b_k\in\R$, $k\in[d]$, with $a_k<b_k$ we define the function
    \[
      g(s) = \E s^{\alpha \mu_t(A)},\quad s\in (0,1],
    \]
    where $A=\prod_{k=1}^d[a_k,b_k)$. We will first prove that 
    \begin{equation}\label{eqn:expansion_for_g}
    g(s) = \sum_{k=0}^N s^k p_k + o(s^N)
    \end{equation}
    as $s \downarrow 0$ for every $N\in\N$ and some sequence $\{p_k\}_{k \in \N_0}$. This fact will imply that $\alpha\mu_t(A)$ can take only non-negative integer values a.s., according to \cite[Lemma~2.4]{konarovskyi2019dean} which is rather an intuitive statement that only the probability generating functions of non-negative integer-valued random variables admit a Taylor's expansion at zero. 
    To check the expansion~\eqref{eqn:expansion_for_g}, we will first show that the function $g$ coincides on $(0,1]$ with an infinitely differentiable function in a neighborhood of zero and then will use Taylor's theorem.
    
    Take a sequence $\{\varphi_n\}_{n\ge 1}\in\C^\infty_c(\R^d)$ such that $\varphi_n\downarrow \I_A$ pointwise as $n\to\infty$. By the dominated convergence theorem, and the estimate $0\le V_t(r\alpha\varphi_n)\le V_t(r\alpha\varphi_1)\in\Schw$ (see Lemmas~\ref{lemma:basic} and~\ref{lemma:technical3}), we get 
    \[
    \av{\mu_t,r\alpha\varphi_n}\to r\alpha\mu_t(A)\quad\mbox{a.s.}\ \ \mbox{and}\quad  \av{\nu,V_t(r\alpha\varphi_n)}\to \av{\nu,V_t(r\alpha\I_A)}
    \]
    as $n\to\infty$ for each $r\ge 0$. Again recall that $V_t \varphi = - \alpha \ln\left(P_t e^{-\frac{\varphi}{\alpha}}\right)$. Thus, using the dominated convergence theorem again and Proposition~\ref{prop:laplace_duality}, we get
    \[
    \E e^{-r\alpha\mu_t(A)}=e^{-\av{\nu,V_t(r\alpha\I_A)}}=e^{\alpha\av{\nu,\ln\left(P_te^{-r\I_A}\right)}}
    \]
    for each $r\ge 0$. Setting $s=e^{-r}$ and using the quality $P_te^{-r\I_A}=1+(s-1)h$ for $h:=P_t\I_A\in \Schw$, we have
    \begin{equation}\label{eqn:form_of_g}
    g(s)=e^{\alpha\av{\nu,\ln(1+(s-1)h)}},\quad s\in(0,1].
    \end{equation}
    Note that there exists $\delta \in (0,1)$ such that $0\le h(x)\le \delta$ for all $x\in\R^d$. We take $\delta_1>0$ satisfying $\gamma:=(1+\delta_1)(1-\delta)<1$ and choose a constant $C>0$ such that $-\ln(1-r)\le C r$ for all $r\in[0,\gamma]$. 
    Then
    \begin{equation*}
      0\le-\ln(1+(s-1)h)\le C(1-s)h\le C(1+\delta_1)h.
    \end{equation*}
    Therefore, the function in the right hand side of~\eqref{eqn:form_of_g} is well defined on $(-\delta,1]$.
    We also note that the derivatives of each order of the function $f(s,\cdot):=\ln(1+(s-1)h)$ are bounded by
    \[
    \left|\frac{\partial^k}{\partial s^k}f(s,x)\right|=(k-1)!\frac{h^k(x)}{\left(1+(s-1)h(x)\right)^k}\le\frac{(k-1)!}{1-\gamma}h^k(x)
    \]
    for all $x\in \R^d$ and $s\in[-\delta_1,1]$. Hence, they are uniformly (in $s\in[-\delta_1,1]$) integrable with respect to the measure $\nu$ because $h^k\in\Schw$. This yields that the function $\tilde g$ is infinitely differentiable in a neighbourhood of zero. Applying Taylor's formula to $\tilde g$ and using the equality \eqref{eqn:form_of_g}, we get the expansion~\eqref{eqn:expansion_for_g}.

Using \cite[Lemma~2.4]{konarovskyi2019dean}, we obtain that $\alpha\mu_t(A)\in\N_0$ for each rectangle $A=\prod_{k=1}^d[a_k,b_k)$ a.s. By the monotone class theorem (see \cite[Theorem~1.1]{FOMP}), it is easy to see that $\mu_t(A)\in\N_0$ for all bounded $A\in\cB(\R^d)$ a.s. Now, using the continuity of $(\mu_t)_{t\ge 0}$ and Lemma~\ref{lemma:Portmanteau}, we can conclude that $\mu_0=\nu$ takes values in $\N_0\cup\{+\infty\}$, and, therefore, it is defined by~\eqref{eqn:form_of_initial_condition}. This completes the proof of the theorem.
\end{proof}

We finish this section with a simple observation that the Laplace duality can be used for the investigation of an invariant measure for the Dean--Kawasaki dynamics on the space of tempered distribution.
\begin{lemma}
    Let $\Xi$ be a Poisson point process. Then the law of $\frac{1}{\alpha} \Xi$ is an invariant measure for the Dean-Kawasaki equation \eqref{eqn:DK}.
\end{lemma}
\begin{proof}
        We first note that the Poisson point process belongs to $\cM_\cS$ a.s. Indeed, taking e.g. $\psi(x)=\frac{1}{1+|x|^{d+1}}$, $x\in\R^d$, and using Campbell’s formula, we get
        \[
          \E\av{\Xi,\psi}=\int_{\R^d}\psi(x)dx<\infty.
        \]
        This implies that $\P\{\Xi\in\cM_{\Schw}\}=\P\left\{\av{\Xi,|\varphi|}<\infty,\  \varphi\in\Schw\right\}=1$ due to the fact that for each $\varphi\in\Schw$ there exists $C>0$ such that $|\varphi|\le C\psi$. Thus, a solution $(\mu_t)_{t\ge 0}$ to~\eqref{eqn:DK} started from $\frac{1}{\alpha} \Xi$ is well-defined, by Lemma~\ref{lem:existence_of_solution}. Our goal is to show that the distributions of $\mu_t$ and $\frac{1}{\alpha} \Xi$ coincide for each $t>0$. 

        By the Laplace transform of a Poisson point process (see~\cite[Secton~1.3]{kallenberg1983random}) and a standard approximation argument, we obtain
        \[
        \E e^{-\av{\frac{1}{\alpha}\Xi, \varphi}}  =e^{-\av{\lambda, 1- e^{-\frac{\varphi}{\alpha}}}}
        \]
        for all $\varphi\in\Schw_+$. Hence, for every $\varphi\in\Schw_+$ Lemma~\ref{prop:laplace_duality} and the disintegration theorem (\cite[Theorems~2.1.7 and~2.2.5]{ito1984}) yield
        \begin{align*}
        \E\left( e^{-\av{\mu_t, \varphi}} \right) &= \E\left( e^{-\av{\mu_0, V_t \varphi}}\right)
        =  e^{ - \av{\lambda,1-e^{-\frac{1}{\alpha} V_t \varphi}} }\\
        &=e^{ - \av{\lambda,1-P_t e^{-\frac{\varphi}{\alpha}}} }
        = e^{ - \av{\lambda, P_t(1-e^{-\frac{\varphi}{\alpha}})} } \\
        &= e^{ - \av{\lambda, 1-e^{-\frac{\varphi}{\alpha}}} }=\E e^{-\av{\frac{1}{\alpha}\Xi, \varphi}}.
\end{align*}
        Applying the same argument as in the proof of Proposition~\ref{prop:uniqueness}, we finish the proof of the statement.
\end{proof}

\appendix

\section{Appendix}

\subsection{Some estimates of solutions to the heat equation and Hamilton-Jacobi equation}
\label{sec:estimates}

We recall that $P_t$ is the heat semigroup whose generator is given by $\frac{\alpha}{2}\Delta$ and $V_t=-\alpha\ln\left(P_te^{-\frac{1}{\alpha}\varphi}\right)$ defines a solution to the Hamilton-Jacobi equation~\eqref{eqn:hamilto_jacobi_equation} with initial condition $\varphi$. In this section, we will collect some properties of the operators $P_t$ and $V_t$ used for the proof of the main result.

\begin{lemma}
	Let $\varphi,\psi \in \C_b(\R^d)$ and $\varphi\le \psi$. Then $V_t \varphi \leq V_t \psi$ for all $t\ge 0$. 
	\label{lemma:basic}
\end{lemma}

\begin{proof}
The statement of the lemma directly follows from the monotonicity of the heat semigroup $P_t$, i.e. $P_t\varphi\ge 0$ for $\varphi\ge 0$.
\end{proof}

\begin{lemma}
    Let $\varphi\in\Schw$. Then $(P_t\varphi)_{t\ge 0}$ and $(\partial_tP_t\varphi)_{t\ge 0}$ belong to $\C([0,\infty),\Schw)$. Moreover, $\P_t|\varphi|\in\Schw$ for each $t>0$.
    \label{lemma:Schw1}
\end{lemma}

\begin{proof}
   The first part of the lemma is, e.g., the statement of \cite[Theorem~5.3]{Hunter:2014}, that is proved via the Fourier transform. The second part of the lemma can be proved similarly, using the fact that the Fourier transform of Schwartz functions is again a Schwartz function. Indeed, let as usually $\widehat\psi$ denote the Fourier transform of an integrable function $\psi$. Since the function $|\varphi|$ is rapidly-decreasing for $\varphi\in\Schw$, we conclude that its Fourier transform $\widehat{|\varphi|}$ is infinitely differentiable whose derivatives are integrable. Thus,
   \[
   \widehat{P_t|\varphi|}(\sigma)=\widehat{|\varphi|}(\sigma) e^{-\frac{\alpha}{2} t|\sigma|^2},\quad \sigma\in\R^d,
   \]
   is smooth rapidly-decreasing function for every $t>0$. This completes the proof of the lemma.
\end{proof}

\begin{corollary}
    For every $T>0$, $\varphi \in \Schw$ and $\mu\in\cM_\Schw$ we have $\sup_{t \in [0,T]} \av{\mu, |P_t \psi|} < \infty$.
    \label{corollary}
\end{corollary}

\begin{proof}
By \cite[Theorem~2.7]{baake2022note}, there exist $n\in\N$ such that 
$
\av{\mu,1/f}<\infty
$
for $f(x)=1+|x|^{2n}$. Thus,
\[
\av{\mu,|P_t\varphi|}\le \av{\mu,1/f}\|fP_t\varphi\|_{0,0}\le \av{\mu,1/f}\left(\|P_t\varphi\|_{0,0}+\|P_t\varphi\|_{0,2n}\right).
\]
Using Lemma~\ref{lemma:Schw1}, we get the needed uniform bound in $t\in[0,T]$.
\end{proof}

We will now prove a similar statement to Lemma~\ref{lemma:Schw1} for $(V_t\varphi)_{t\ge 0}$.

\begin{lemma}
	Let $\varphi \in \mathcal{S}$. Then $(V_t\varphi)_{t\ge 0}$ and $(\partial_t V_t\varphi)_{t\ge 0}$ belong to $\C([0,\infty),\Schw)$. In particular, $(\Delta V_t\varphi)_{t\ge 0}$ and $(\Gamma (V_t\varphi))_{t\ge 0}$ belong to $\C([0,\infty),\Schw)$. 
	\label{lemma:technical3}
\end{lemma}

\begin{proof}
Let 
\[
\Schw_{>-1}=\left\{\varphi\in\Schw:\ \varphi(x)>-1\ \ \mbox{for all}\ \ x\in\R^d\right\}
\]
be a subspace of $\Schw$ with the induced topology. Define the maps 
\begin{equation}\label{eqn:operators_L_and_H}
\begin{split}
    \cL(\varphi)&:=e^{\frac{1}{\alpha}\varphi}-1,\quad \varphi\in\Schw,\\
    \cH(\varphi)&:=-\alpha\ln(1+\varphi),\quad \varphi\in\Schw_{>-1}.
\end{split}
\end{equation}
Then a direct computations show that $\cL:\Schw\to \Schw_{>-1}$ and $\cH:\Schw_{>-1}\to\Schw$ are continuous. Therefore, the continuity of $(V_t\varphi)_{t\ge 0}$ directly follows from Lemma~\ref{lemma:Schw1} and the equality fact that 
\[
V_t\varphi=-\alpha\ln\left(1+P_t\left(e^{-\frac{1}{\alpha}\varphi}-1\right)\right)=(\cH\circ P_t\circ\cL)(\varphi)
\]
for all $t\ge 0$. Since $\Delta:\Schw\to\Schw$ and $\Gamma:\Schw\to\Schw$ are continuous in $\Schw$, it implies the second part of the lemma. The continuity of $(\partial V_t\varphi)_{t\ge 0}$ follows from the equality
\begin{equation}\label{eqn:equality_for_partial_t_V}
\partial_tV_t\varphi=\frac{\alpha}{2}\Delta V_t\varphi+\Gamma (V_t\varphi)
\end{equation}
for all $t\ge 0$. This completes the proof of the lemma.
\end{proof}

\begin{lemma}
Let  $\kappa\in\Schw$ be a strictly positive functions such that $\kappa(x)=e^{-|x|}$ for all $x\in\R^d$ with $|x|>1$. Then for each $\varphi \in \C_c^\infty(\R^d)$ and $T>0$ there exists a constant $C>0$ such that 
	\begin{align*}        |V_t\varphi|+|\partial_t V_t\varphi|+|\Delta V_t\varphi|+|\Gamma(V_t\varphi)|\le C\kappa
	\end{align*}
 for all $t\in[0,T]$.
	\label{lemma:technical1}
\end{lemma}

\begin{proof}
    Let $T>0$ be fixed.
    Denote by $p_t$ the distribution of a Brownian motion on $\R^d$ with diffusion rate $\alpha$ started from zero. Then for each $\varphi\in\C_c^\infty(\R^d)$ there exists a constant $\tilde C>0$ such that $|\varphi|\le p_1$. Thus,
    \[
    |P_t\varphi|\le \tilde C P_tp_1=\tilde C p_{t+1}\le C \kappa 
    \]
    for all $t\in[0,T]$ and some constant $C_{\varphi,T}>0$ depending on $\varphi$.
    Therefore, using integration by parts, we get for each $k\in \N$, $j\in[d]$ and $t\in[0,T]$
    \[    
    |\partial_j^kP_t\varphi|=|P_t(\partial_j^k\varphi)|\le C_{\partial_j^k\varphi,T}\kappa.
    \]

    Let the operators $\cL$ and $\cH$ be defined by \eqref{eqn:operators_L_and_H}. Since $\cL(\varphi)\in\Schw_{-1}$, there exists $\delta>0$ such that $\inf_{x\in\R^d}\cL(\varphi)(x)\ge -1+\delta$. Then choosing a constant $\tilde C$ such that 
    \[
    |\alpha\ln(1+r)|\le \tilde C |r|,\quad r\in[-1+\delta,\infty),
    \]
    observing that $\cL(\varphi)\in\C^\infty_c(\R^d)$ and using the maximum principle  for $P_t$, we obtain
    \[
    |V_t\varphi|=\left|(\cH\circ P_t\circ \cL)(\varphi)\right|\le \tilde C\left|P_t\cL(\varphi)\right|\le\tilde C C_{\cL(\varphi),T}\kappa
    \]
    for all $t\ge 0$. Similarly,
    \begin{align*}
        |\partial_j V_t\varphi|=\alpha\frac{\left|\partial_j P_t\cL(\varphi)\right|}{1+P_t\cL(\varphi)}\le \frac{\alpha}{\delta}C_{\partial_j\cL(\varphi),T}\kappa 
    \end{align*}
    and
    \begin{align*}        \left|\partial_j^2V_t\varphi\right|\le\alpha\frac{\left|\partial_j^2 P_t\cL(\varphi)\right|}{1+P_t\cL(\varphi)}+\alpha\frac{\left(\partial_j P_t\cL(\varphi)\right)^2}{(1+P_t\cL(\varphi))^2}\le\alpha\left[\frac{C_{\partial_j^2\cL(\varphi)}}{\delta}+\frac{C_{\partial_j\cL(\varphi),T}^2C'}{\delta^2}\right]\kappa
    \end{align*}
    for all $t\in[0,T]$ and $j\in[d]$, where we have used the inequality $\kappa^2 \le C'\kappa$ for a constant $C'>0$ in the last step. This implies that there exists a constant $C>0$ such that 
    \[
    |V_t\varphi|+|\Gamma( V_t\varphi)|+|\Delta V_t\varphi|\le C\kappa
    \]
    for all $t\in[0,T]$. Taking into account that $(V_t\varphi)_{t\ge 0}$ solves the Hamilton-Jacobi equation~\eqref{eqn:hamilto_jacobi_equation}, we get the bound for $\partial_t V_t\varphi$. This completes the proof of the lemma.
\end{proof}

\subsection{Portmanteau theorem for tempered measures}
\label{sec:helpful}
We need following version of the Portmanteau theorem, adapted to tempered measures. Note that the statement is not true for unbounded sets.

\begin{lemma}
	Let $\{\mu_n\}_{n\ge 1}$ be a sequence converging to $\mu$ in $\cM_\Schw$ and $A\in\cB(\R^d)$ be a  bounded set satisfying $\mu(\partial A) = 0$.
	Then, $\lim_{n \to \infty} \mu_n(A) = \mu(A)$.
	\label{lemma:Portmanteau}
\end{lemma}
\begin{proof}
    Let $\Phi\in\C^\infty_c(\R^d)$ be a fixed non-negative function such that $\Phi=1$ on $A$. If $\av{\mu,\Phi}=0$, then $\mu(A)=0$ and
    \[
    0\le\limsup_{n\to \infty}\mu_n(A)\le \lim_{n\to \infty}\av{\mu_n,\Phi}=0
    \]
    that gives the statement of the lemma. Now, let $\av{\mu,\Phi}>0$. Without loss of generality, we assume that $\av{\mu_n,\Phi}>0$ for all $n\ge 1$. Define the probability measures
    \[    \tilde\mu_n(B)=\frac{1}{\av{\mu_n,\Phi}}\int_B\Phi(x)\mu_n(\dd x),\quad B\in\cB(\R^d),
    \]
    for all $n\ge 1$ and
    \[    
    \tilde\mu(B)=\frac{1}{\av{\mu,\Phi}}\int_B\Phi(x)\mu(\dd x),\quad B\in\cB(\R^d),
    \]
    that are supported on $\supp\Phi$. Therefore, $\av{\tilde\mu_n,\varphi}\to\av{\tilde\mu,\varphi}$ as $n\to\infty$ for each $\varphi\in\C^\infty_b(\R^d)$. This implies that $\tilde\mu_n\to\tilde\mu$ in distribution. By the absolute continuity of $\tilde\mu_n$ with respect to $\mu_n$ we have that $\tilde\mu_n(\partial A)=0$ for all $n\ge 1$. Thus, by \cite[Theorem~3.3.1]{ethierkurtz},
    \[    \lim_{n\to\infty}\mu_n(A)=\lim_{n\to\infty)}\av{\mu_n,\Phi}\tilde\mu_n(A)=\av{\mu,\Phi}\tilde\mu(A)=\mu(A).
    \]
    This completes the proof of the lemma.
\end{proof}

\section*{Acknowledgements}
The first author was partially supported by the Deutsche Forschungsgemeinschaft (DFG, German Research Foundation) - SFB 1283/2 2021 - 317210226. 

The first author thanks the Max Planck Institute for Mathematics in the Sciences and the University of Bielefeld for their warm hospitality, where a part of this research was carried out. The second author thanks Daniel Heydecker for his discussions of the Example~\ref{exa:blow_up} (see also the Coffee Break Problem on his web page\footnote{\url{https://danielheydecker.wordpress.com/2022/09/19/coffee-break-problem-19-09-2022/}}). The both authors are grateful to Max von Renesse for many fruitful discussions.  

\bibliography{lib_mct}

\end{document}